\documentclass[11pt, a4paper,leqno]{amsart}
\usepackage{amsmath,amsthm,amscd,amssymb,amsfonts, amsbsy}
\usepackage{latexsym}
\usepackage{exscale}
\usepackage{txfonts}
\usepackage[colorlinks,citecolor=red,pagebackref,hypertexnames=false]{hyperref}

\usepackage{pgf}
\usepackage{color}

\numberwithin{equation}{section}

\theoremstyle{plain}
\newtheorem{theorem}[equation]{Theorem}
\newtheorem{lemma}[equation]{Lemma}

\theoremstyle{definition}
\newtheorem{definition}[equation]{Definition}

\theoremstyle{remark}
\newtheorem{remark}[equation]{Remark}

\newcommand{\dv}{\operatorname{div}}

\newcommand{\supp}{\operatorname{supp}}
\newcommand{\dist}{\operatorname{dist}}

\newcommand{\lc}{\mathcal{L}}

\newcommand{\bX}{{\bf X}}
\newcommand{\bY}{{\bf Y}}
\newcommand{\bx}{{\bf x}}
\newcommand{\by}{{\bf y}}
\newcommand{\bz}{{\bf z}}

\newcommand{\re}{\mathbb{R}}
\newcommand{\rn}{\mathbb{R}^n}

\newcommand{\ree}{\mathbb{R}^{n+1}}

\newcommand{\eps}{\varepsilon}

\newcommand{\p}{\mathcal{P}}

\newcommand{\Hc}{\mathcal{H}}
\newcommand{\C}{\mathcal{C}}

\newcommand{\A}{\mathcal{A}}

\newcommand{\po}{{\partial\Omega}}

\newcommand{\ppo}{\mathcal{P}\Omega}
\newcommand{\eo}{\partial_e\Omega}
\newcommand{\no}{\partial_n\Omega}
\newcommand{\so}{\partial_s\Omega}
\newcommand{\sso}{\partial_{ss}\Omega}
\newcommand{\ao}{\partial_a\Omega}
\newcommand{\bo}{\mathcal{B}\Omega}
\newcommand{\So}{\mathcal{S}\Omega}

\newcommand{\hm}{\omega}

\newcommand{\pom}{\partial\Omega}

\DeclareMathOperator{\diam}{diam}

\def\Yint#1{\mathchoice
    {\YYint\displaystyle\textstyle{#1}}%
    {\YYint\textstyle\scriptstyle{#1}}%
    {\YYint\scriptstyle\scriptscriptstyle{#1}}%
    {\YYint\scriptscriptstyle\scriptscriptstyle{#1}}%
      \!\iint}
\def\YYint#1#2#3{{\setbox0=\hbox{$#1{#2#3}{\iint}$}
    \vcenter{\hbox{$#2#3$}}\kern-.51\wd0}}
\def\longdash{{-}\mkern-2.0mu{-}} 
  
\def\tiltlongdash{\rotatebox[origin=c]{12}{$\longdash$}}

\def\tiltfiint{\Yint\tiltlongdash}

\begin{document}

\title[BMO Solvability]{BMO Solvability and Absolute Continuity of Caloric Measure}

\author{Alyssa Genschaw}

\address{Alyssa Genschaw
\\
Department of Mathematics
\\
University of Missouri
\\
Columbia, MO 65211, USA} \email{adcvd3@mail.missouri.edu}

\author{Steve Hofmann}

\address{Steve Hofmann\\
Department of Mathematics\\
University of Missouri\\
Columbia, MO 65211, USA} \email{hofmanns@missouri.edu}

\thanks{The authors were supported by 
NSF grant number DMS-1664047.}

\date{\today}
\subjclass[2000]{42B99, 42B25, 35J25, 42B20}

\keywords{BMO,  Dirichlet problem, caloric measure, parabolic measure, 
divergence form parabolic equations, 
weak-$A_\infty$, Ahlfors-David 
Regularity 
}

\begin{abstract} 
We show that BMO-solvability implies 
scale invariant quantitative absolute continuity (specifically, the weak-$A_\infty$ property) of 
caloric measure with respect to surface measure,  for an open set $\Omega \subset \ree$,
assuming as a background hypothesis only that the essential boundary
of $\Omega$ satisfies an appropriate parabolic version
 of Ahlfors-David
regularity, entailing some backwards in time thickness.  Since the weak-$A_\infty$ property of the caloric
measure is equivalent to $L^p$ solvability of the initial-Dirichlet problem, we may then deduce that
$BMO$-solvability implies $L^p$ solvability for some finite $p$.
\end{abstract}

\maketitle

{\small
\tableofcontents}

\section{Introduction}\label{s1}  
In the setting of divergence form elliptic PDE, it is well known that
solvability of the Dirichlet problem
with $L^p$ data is equivalent to
scale-invariant absolute continuity of elliptic-harmonic measure
(specifically that elliptic-harmonic measure belongs to the Muckenhoupt weight class
$A_\infty$ with respect to surface measure on the boundary).
To be more precise, in a Lipschitz or even chord-arc domain,
one obtains that the Dirichlet problem is solvable with data
in $L^p(\Omega)$ for some $1<p<\infty$, if and only if elliptic-harmonic measure $\omega$ with some fixed pole is absolutely
continuous with respect to surface measure $\sigma$ on the boundary, and the Poisson
kernel $d\omega/d\sigma$ satisfies a reverse H\"older condition with exponent $p'=p/(p-1)$;
see the monograph of Kenig \cite{Ke}, and the references cited there.   In fact, the equivalence between
$L^p$ solvability and quantitative absolute continuity holds much more generally, for any open set
with an Ahlfors-David regular boundary (see \cite{HLe} for a proof, although the result is somewhat folkloric); 
in this generality, the $A_\infty$/reverse-H\"older property is 
(necessarily) replaced by
its weak version, which does not entail doubling.

These results have endpoint versions, as well: in \cite{DKP},
Dindos, Kenig and Pipher showed that in a Lipschitz domain (or even a chord-arc domain)
elliptic-harmonic measure satisfies
an $A_\infty$ condition
with respect to surface measure, if and only if a natural
Carleson measure/BMO estimate
 holds for solutions of the Dirichlet problem with continuous data.
The results of \cite{DKP} have been extended to the setting of a 1-sided Chord-arc domain
by Z. Zhao \cite{Z}.

In the above works, the proofs relied substantially 
on quantitative connectivity of the domain, in the 
form of the Harnack Chain condition. 
More recently, the second named author and P. Le \cite{HLe} 
proved an analogous result in the absence of any connectivity hypothesis, either quantitative or qualitative:  one obtains
that BMO solvability implies scale invariant quantitative absolute continuity (the weak-$A_{\infty}$ property) 
of elliptic-harmonic measure with respect to surface measure on $\partial\Omega$, assuming only that $\Omega$ 
is an open set with Ahlfors-David regular boundary\footnote{A partial converse is also obtained in 
\cite{HLe}: namely that in the special case of the Laplace operator,  the weak-$A_\infty$ property of harmonic measure
implies BMO solvability, assuming in addition that the open set $\Omega$ satisfies an interior Corkscrew condition.}. 

The goal of the present paper is to extend the results of \cite{HLe} to the parabolic setting.  As regards geometric hypotheses,
we assume only that $\Omega \subset \ree$ is an open set whose boundary satisfies an appropriate version
of a parabolic Ahlfors-David regularity condition.
In particular, we impose no connectivity hypothesis, such as a parabolic Harnack chain condition.

We shall consider 
the heat operator
\begin{equation}
\label{eq1.1a}
L_0:=\partial_t- \lc,
\end{equation}
where $\lc := \nabla\cdot\nabla$ is the usual Laplacian in $\rn$, acting in the space variables.
In some circumstances, to be discussed momentarily, 
our results apply more generally to divergence form parabolic operators 
\begin{equation}
\label{eq1.1}
L:=\partial_t- \dv A(X,t) \nabla,
\end{equation}
defined in an open set $\Omega\subset\mathbb{R}^{n+1}$  as described above,
where $A$ is $n\times n$, real,  $L^\infty$, and satisfies the uniform ellipticity condition
\begin{equation}
\label{eq1.1*} \lambda|\xi|^{2}\leq\,\langle A(X,t)\xi,\xi\rangle := \sum_{i,j=1}^{n+1}A_{ij}(X,t)\xi_{j} \xi_{i}, \quad
  \Vert A\Vert_{L^{\infty}(\mathbb{R}^{n})}\leq\lambda^{-1},
\end{equation}
 for some $\lambda>0$, and for all $\xi\in\rn$, and a.e. $(X,t)\in \Omega$.  
We do not require that the matrix $A(X,t)$ be symmetric.   

More precisely, our results will apply to variable coefficient parabolic operators as in
 \eqref{eq1.1}, provided that the continuous Dirichlet problem (see Definition \ref{bvpdef} - I below), 
 is solvable in $\Omega$
 (and hence that parabolic measure for $L$ can be defined).  Beyond the class of constant
 coefficient parabolic operators, such a solvability result holds when the coefficients are $C^1$-Dini:
 indeed, we shall impose an appropriate parabolic version of Ahlfors-David regularity 
 which in particular  implies the capacitary Wiener criterion, valid in the case of $C^1$-Dini coefficients,
proved by Fabes, Garofalo and Lanconelli \cite{FGL}.

Before stating our main theorem, we briefly
introduce some of the concepts and notation
to be used.  All additional terminology used in 
the statement of the theorem, and not  discussed here or above, 
will be defined precisely in the sequel.  For now, we note that all distances and diameters
are taken with respect to the parabolic distance \eqref{pardist}, and that $\delta(X,t) := \dist((X,t),\eo)$,
where $\eo$ denotes the {\em essential boundary} 
(see Definition \ref{parabolicb} below) of an open set $\Omega\subset \ree$.  
We further note that ``surface measure" 
$\sigma$ on the {\em quasi-lateral boundary}\footnote{See Definition \ref{parabolicb}.
In the present work, the ADR condition that we impose will imply that the quasi-lateral 
boundary is a natural substitute for the lateral boundary (in cylindrical domains, for example, they are the same).  
See also Remarks
\ref{r-adr} and \ref{r-adr2}.} 
$\Sigma$, is defined by
$d\sigma=d\sigma_{\!s}ds$, 
where $d\sigma_{\!s}:=\Hc^{n-1}\vert_{\Sigma_s}$,
the restriction of 
$(n-1)$-dimensional Hausdorff measure to the time slice $\Sigma_s := \Sigma\cap \{t\equiv s\}$.
We let $T_{min}, T_{max}$ denote, respectively, the smallest and largest values of the time co-ordinate
occurring in $\overline{\Omega}$; see \eqref{tminmax}.

We note that for an arbitrary open set $\Omega\subset \ree$, caloric measure
may be constructed via the PWB method, since continuous functions on the essential boundary
are resolutive;  see \cite{W1} or \cite[Chapter 8]{W2}.

Given an open set $\Omega\subset \ree$ 
 and a divergence form parabolic operator $L$ as above, for which the continuous Dirichlet problem is solvable,
 we shall say that 
the  initial-Dirichlet problem 
(see Definition \ref{bvpdef} below) is \textit{BMO-solvable}\footnote{Perhaps ``VMO-solvable" would be a 
more appropriate term, but ``BMO-solvable" seems to be entrenched in the literature.  In less austere settings, 
the two notions are equivalent, at least in the elliptic case; see \cite[Remark 4.20]{HLe}.} for 
$L$ in $\Omega$ if for all continuous $f$ with 
compact support on $\Sigma$, the solution $u$ of the  initial-Dirichlet problem with data $f$ 
satisfies the Carleson measure estimate
\begin{multline}\label{1.3}
\underset{(x,t)\in \Sigma, \,0<r<R(t)}\sup 
\,\,\dfrac{1}{\sigma\left(\Delta_r\right)}\iint_{\Omega\cap Q_r(x,t)}\left(|\nabla u(Y,s)|^2+|\delta(Y,s)\partial_s u(Y,s)|^2\right)\delta(Y,s)\,dYds \\[4pt]
\leq \, C||f||_{BMO(\Sigma)}^2,
\end{multline}
where 
 $\Delta_r\!:=Q_r(x,t)\cap \Sigma$, and $R(t) \!:= \!\min\!\Big(R_0,\sqrt{t-T_{min}}/\big(4\!\sqrt{n}\,\big)\Big)$, with
$R_0\!:=\diam(\Sigma)$.
We recall that $T_{max}-T_{min} \gtrsim R^2_0$; see the discussion
preceeding \cite[Theorem 2.9]{GH}.

For $(X,t)\in \Omega$,
 we let
 $\hm_L^{X,t}$  denote parabolic measure for $L$ with pole at $(X,t)$, and if the dependence on
 $L$ is clear in context, we shall simply write $\hm^{X,t}$.

The main result of this paper is the following.  All terminology used in the statement of the theorem and not 
discussed
already, will be defined precisely in the sequel.

\begin{theorem}\label{tmain} Let $L$ be a divergence form parabolic operator defined on 
$\Omega$.  
Let $\Sigma$ be globally time-backwards ADR, 
and assume further that if $R_0:= \diam \Sigma = \infty$, then $T_{min} = -\infty$.

If the initial-Dirichlet problem for $L$ is BMO-solvable in $\Omega$, then the parabolic measure 
belongs to weak-$A_{\infty}$ in the following sense: for every parabolic cube $Q:=Q_r(x_0,t_0)$, with $(x_0,t_0)\in \Sigma$ and 
$0<r<\min\left(R_0,\sqrt{t_0-T_{min}}/\big(4\!\sqrt{n}\,\big)\right)$, 
and for all $(Y,s)\in \Omega\setminus 4Q$, parabolic measure $\omega_L^{Y,s}\in$ weak-$A_{\infty}(\Delta)$, 
where the parameters in the weak-$A_{\infty}$ condition are uniform in $\Delta$.
\end{theorem}

We note that we are implicitly assuming here, as above, that the continuous Dirichlet problem is solvable for
$L$;  we know that this is true if $L$ is the heat operator, or if the coefficients of $L$ are $C^1$-Dini: see Remarks \ref{r1.22} and \ref{r-adr2}.   We note that our assumption of solvability of the continuous Dirichlet problem is used only qualitatively:
the constants in our estimates will depend only on dimension, ellipticity, and the constant in the BMO-solvability
estimate \eqref{1.3}.

\begin{remark} By \cite[Theorem 2.9]{GH}, the weak-$A_\infty$ property of caloric (or parabolic) measure
is equivalent to $L^p$ solvability of the initial-Dirichlet problem (see \cite{GH} for a precise statement), 
for some $p<\infty$; hence the latter 
also follows from BMO solvability.
\end{remark}

\begin{remark} In the elliptic case, the analogue of Theorem \ref{tmain} has a partial converse
\cite[Theorem 1.6]{HLe}, valid
for the Laplacian: under the additional assumption that $\Omega$ satisfies an interior Corkscrew condition,
if $\pom$ is ADR, and harmonic measure belongs to weak-$A_\infty$ with respect to surface measure on $\pom$, then
the Dirichlet problem for Laplace's equation is BMO-solvable.  In the parabolic setting,
this converse result remains open.  The proof in the elliptic case relies on 
square-function/non-tangential-maximal-function estimates, which in turn are obtained by invoking
results of \cite{HM} (see also \cite{HLMN}, \cite{MT}) to deduce uniform rectifiability of $\pom$;
see \cite{HM}, \cite{HMM1}, \cite{HMM2} (as well as \cite{GMT}, \cite{AGMT} for related converse results).  
The machinery created in these references, and exploited in \cite{HLe}, has yet to be developed in the parabolic setting.

\end{remark}

The paper is organized as follows.  In the remainder of this section, 
we present some basic notations and definitions.  In Section \ref{s2}, we state two lemmas and a corollary which we then use to prove Theorem \ref{tmain}. In Section \ref{s3} we prove Theorem \ref{tmain}.

\noindent{\bf Notation and Definitions}
For a set $A\subset \ree$, we define
\begin{equation}\label{tminmax}
T_{min}(A):=\inf\{T: A\cap\{t\equiv T\} \neq \emptyset\}\,,\quad T_{max}(A):=\sup\{T: A\cap\{t\equiv T\} \neq \emptyset\}
\end{equation}
(note:  it may be that $T_{min}(A) = -\infty$, and/or that $T_{max}(A) = +\infty$).  In the special case that $A=\Omega$, an open set that
has been fixed, we will simply write
$T_{min} = T_{min}(\Omega)$ and $T_{max}=T_{max}(\Omega)$.

\begin{definition}[{\bf Parabolic cubes}]\label{paraboliccube}
An (open) parabolic cube in $\rn \times \re$ with  center
$(X,t)$:
\begin{multline}\label{eqcubedef}
Q_r(X,t):=Q((X,t),r)\\:=\{(Y,s)\in \rn \times \re: |X_i-Y_i|<r \,,  \, 1\leq i \leq n, \,
t-r^2<s < t+r^2\}.
\end{multline}
With a mild abuse of terminology, we refer to $r$ as the ``parabolic sidelength" (or simply the ``length")
of $Q_r(X,t)$.
We shall sometimes simply write $Q_r$ to denote a cube of parabolic length $r$, when the center is implicit,
and for $Q=Q_r$, we shall write $\ell(Q) = r$.

We also consider the time-backward and time-forward versions:  
\begin{multline*}
Q^-((X,t),r):=Q_r^-(X,t)\\:=\{(Y,s)\in \rn\times \re: |X_i-Y_i|<r\,,\, 
1\leq i\leq n\,, \, t-r^2<s< t\},
\end{multline*}
\begin{multline*}
Q^+((X,t),r):=Q_r^+(X,t)\\:=\{(Y,s)\in \rn\times \re : |X_i-Y_i|<r \,,\,
1\leq i\leq n\,,\, t< s < t+r^2\}\,.
\end{multline*}
We shall sometimes also use the letter $P$ to denote parabolic cubes in $\ree$.
\end{definition}

\begin{definition}[{\bf Classification of boundary points}]\label{parabolicb}
Following \cite{L}, given an open set $\Omega\subset \ree$, we define its {\em parabolic
boundary} $\p\Omega$ as
$$\mathcal{P}\Omega:=\left\{(x,t)\in\partial\Omega: \forall r>0 \,, \,Q_r^-(x,t)\,
\text{ meets }\, \ree\setminus \Omega\right\}.
$$
 The {\em bottom boundary}, denoted $\bo$, is defined as 
 \[\bo:= \left\{(x,t)\in\ppo: \exists \,\eps>0 \,  \text{ such that } \,Q_\eps^+(x,t)\,
\subset \Omega\right\}.\]
The {\em lateral boundary}, denoted $\So$, is defined as
$\So :=\ppo\setminus \bo$.

 Following \cite{W1,W2}, we also define the {\em normal boundary}, denoted
 $\no$, to be equal to the parabolic boundary in a bounded domain, while in an unbounded domain, we 
 append the point at infinity:
 $\no = \ppo\cup\{\infty\}$.   The {\em abnormal boundary} is defined as $\ao:= \pom\setminus \no$, thus:
  \[\ao:= \left\{(x,t)\in\partial\Omega: \exists \,\eps>0 \,  \text{ such that } \,Q_\eps^-(x,t)\,
\subset \Omega\right\}.\]
The abnormal boundary is further decomposed into $\ao=\so\cup\sso$ (the
{\em singular boundary} and {\em semi-singular boundary}, respectively), where
\[\so:=  \left\{(x,t)\in\partial_a\Omega: \exists \,\eps>0 \,  \text{ such that } \,Q_\eps^+(x,t)\,
\cap \Omega=\emptyset\right\},\]
and
\[\sso:=  \left\{(x,t)\in\partial_a\Omega: \forall \,r>0 \,  \,Q_r^+(x,t)\,\text{ meets }\,
\Omega\right\}.\]
The {\em essential boundary} 
$\eo$, is defined as
\begin{equation}\label{sigma}
\eo\,:=\, \no\,\cup\, \sso \,=\, \partial\Omega \setminus \so
\end{equation}
(where we replace $\pom$ by $\pom\cup\{\infty\}$ if $\Omega$ is unbounded).
Finally, 
 we define the {\em quasi-lateral boundary} $\Sigma$ to be
\begin{equation}\label{sigmazero}
\Sigma:= \,\left\{
\begin{array}{l}
\po \,,\quad  \text{if } T_{min} = -\infty \, \text{ and } T_{max} = \infty\\[4pt]
\po \setminus (\bo)_{T_{min}}\,, \quad  \text{if } T_{min} > -\infty \, \text{ and } T_{max} = \infty\\[4pt]
\po\setminus (\so)_{T_{max}} \,, \quad  \text{if } T_{max} <\infty \, \text{ and } T_{min} = - \infty
 \\[4pt]
 \po\setminus \left( (\bo)_{T_{min}} \cup (\so)_{T_{max}}\right) \,, 
 \quad  \text{if }\, -\infty<T_{min}<T_{max} <\infty \,.
 \end{array}\right.
\end{equation}
where $(\bo)_{T_{min}}$ is the 
time slice of $\bo$ with $t\equiv T_{min}$,
and $(\so)_{T_{max}}$ is the time slice of $\so$ with $t \equiv T_{max}$. 
Observe that for a cylindrical domain
$\Omega= U \times (T_{min},T_{max})$, with $U\subset \rn$ a domain in the spatial variables,  
then $\Sigma$ would simply be the usual lateral boundary.
\end{definition}

 Caloric measure is supported on the essential boundary; see
 \cite{Su}, or  \cite{W1,W2}.
 
 \begin{list}{$\bullet$}{\leftmargin=0.4cm  \itemsep=0.2cm}
\item We use the letters $c,C$ to denote harmless positive constants, not necessarily
the same at each occurrence, which depend only on dimension and the
constants appearing in the hypotheses of the theorems (which we refer to as the
``allowable parameters'').  We shall also
sometimes write $a\lesssim b$ and $a \approx b$ to mean, respectively,
that $a \leq C b$ and $0< c \leq a/b\leq C$, where the constants $c$ and $C$ are as above, unless
explicitly noted to the contrary.  
\item We shall
use lower case letters $x,y,z$, etc., to denote the spatial component of points on the boundary $\pom$, 
and capital letters
$X,Y,Z$, etc., to denote the spatial component of generic points in $\ree$ (in particular those in $\Omega$).
\item For the sake of notational brevity, we shall sometimes also use boldface capital letters
to denote points in space time $\ree$, and lower case boldface letters to denote points on $\pom$;  thus,
\[ \bX = (X,t) , \quad \bY = (Y,s) , \quad
\text{and} \quad  \bx = (x,t) , \quad \by = (y,s) ,\]

\item We shall orient our coordinate axes so that time runs from left to right.
\item We let $\mathbb{S}^n$ denote the unit sphere in $\ree$.
\item $\Hc^{d}$ denotes $d$-dimensional Hausdorff measure.
\item For $A\subset \ree$, let $A_{s}:=\{(X,t)\in A: t\equiv s\}$ denote the time slice of $A$ with $t\equiv s$. 
\item We  let
$d\sigma=d\sigma_{\!s}ds$ denote the ``surface measure'' on the quasi-lateral
 boundary
$\Sigma$, where $d\sigma_{\!s}:=\Hc^{n-1}\vert_{\Sigma_s}$, and $\Sigma_s$ is the time slice of $\Sigma$, with
$t\equiv s$.

\item The parabolic norm of ${\bf X}=(X,t)\in \ree$, denoted $\|{\bf X}\|$, is the unique solution
$\rho$ of the equation
$$\frac{|X|^2}{\rho^2} + \frac{t^2}{\rho^4} = 1\,.$$
We observe that the parabolic norm satisfies
\begin{align}\label{pardist}
\|\bX\|=||(X,t)||\approx |X|+|t|^{1/2}.
\end{align}
The parabolic $\ell^\infty$ norm is defined by
\begin{equation}\label{pardistinfty}
\|\bX\|_{\ell^\infty}=||(X,t)||_{\ell^\infty}:= \max\left\{|X_1|,...,|X_n| , |t|^{1/2}\right\}\,.
\end{equation}
Of course, the parabolic norm and parabolic $\ell^\infty$ norm induce corresponding distances
on $\ree$, which are comparable to each other.
\item If $\bX\in \Omega$, we let $\delta(\bX):=\text{dist}(\bX, \eo),$ and
$\delta_\infty(\bX):=\text{dist}_\infty(\bX, \eo),$ denote the 
parabolic distance, respectively, the parabolic $\ell^\infty$ distance, to the essential boundary.
We note that $\delta_{\infty}(X,t) \approx \delta(X,t)$,
with uniform implicit constants depending only on dimension.   
\end{list}

We shall find it convenient to work with ``touching cubes" and ``touching points" with respect to
the parabolic
$\ell^\infty$  distance to the essential boundary: 
\begin{definition}\label{def-touch}
Given a point $(X,t)\in \Omega$,  let $Q_{\star}(X,t)$ denote  the ``touching cube" for the point $(X,t)$, i.e.,
set  $Q_{\star}(X,t):= Q_r(X,t)$, where
$$r=r_{\star}(X,t) := \sup\left\{\rho >0 : Q_\rho(X,t) \cap \partial_e\Omega =\emptyset\right\}\,,$$
 so that (since our cubes are open)
 $Q_{\star}(X,t) \cap \partial_e\Omega =\emptyset$, and $\partial Q_{\star}(x,t)$ meets $\partial_e\Omega$.
We shall say that 
$(\hat{x},\hat{t})\in\partial_e \Omega$ is a ``touching point" for $(X,t)$,
if $(\hat{x},\hat{t})\in\partial Q_{\star}(x,t) \cap \partial_e \Omega$.
Note that $\delta_{\infty}(X,t):= r_{\star}(X,t)$.

 We further note that  if $\sqrt{t-T_{min}} >\delta_\infty(X,t)$, then
 $(\hat{x},\hat{t})\in \Sigma$, for any touching point 
 $(\hat{x},\hat{t})$ of $(X,t)$, i.e., in this case
 $\delta_\infty(X,t) = \dist_\infty((X,t),\Sigma)$.
 \end{definition}

\begin{list}{$\bullet$}{\leftmargin=0.4cm  \itemsep=0.2cm}
\item For a set $A\subset \ree$, we shall write $\diam(A)$ to denote the 
diameter of $A$ with respect to the parabolic distance, i.e.,
\begin{equation*}
\diam(A) := \sup_{\left((X,t),(Y,s)\right)\in A\times A}\|(X,t) - (Y,s)\|\,.
\end{equation*}

\item Given a Borel measure $\mu$, and a Borel set $A\subset \rn$, with positive and finite $\mu$ measure, we
set $\fint_A f d\mu := \mu(A)^{-1} \int_A f d\mu$;  if $A$ is a subset of space-time $\ree$, we then write
$\tiltfiint_A f d\mu := \mu(A)^{-1} \iint_A f(X,t)\, d\mu(X,t)$.

\item A ``surface cube" on $\Sigma$ is defined by
$$\Delta =Q\cap \Sigma\,,$$
where $Q$ is a parabolic cube centered on $\Sigma$,
or more precisely,
$$\Delta = \Delta_r(x,t):= Q_r(x,t)\cap\Sigma\,,$$
with $(x,t)\in \Sigma$.  We note that the
``surface cubes'' are not the same as the dyadic cubes
of M. Christ \cite{Ch}
 on $\Sigma$; we apologize to the reader for the possibly confusing terminology.

\end{list}
 
 \begin{definition}\label{bvpdef}
We define the following boundary value problems.  
The second
is relevant only in the case that $T_{min} = -\infty$.

 \begin{list}{}{\leftmargin=0.4cm  \itemsep=0.2cm}

   \item I. Continuous Dirichlet Problem:   
 \begin{center}
$(D)\left\{ \begin{array}{rl}Lu&\!\!=0  \textrm{ in } \Omega  \\
u\vert_{\eo}&\!\!=f \in C_c(\eo)
\\ u&\!\!\in C(\Omega\cup\no)\,.
\end{array} \right.$
\end{center} 
If $\Omega$ is unbounded, 
we further specify that $\lim_{\|\bX\|\to\infty}
u(\bX)=0$.
Here,  we interpret the statement $u\vert_{\eo}=f \in C_c(\eo)$
to mean that 
 \[\lim_{(X,t)\to (y,s)} u(X,t) = f(y,s)\,,\quad (y,s) \in \no\,,\]
and
 \[\lim_{(X,t)\to (y,s^+)} u(X,t) = f(y,s)\,,\quad (y,s) \in \sso\,.\]
If  the preceeding problem is solvable for all $f\in C_c(\eo)$, then we say that the
``continuous Dirichlet problem is solvable for $L$."

 \item II. $L^p$ Dirichlet Problem: 
 \begin{center}
$(D)_p\left\{ \begin{array}{rl}Lu&=0  \textrm{ in } \Omega  \\
u\vert_{\Sigma}&=f\in L^p(\Sigma)
\\&N_* u\in L^p(\Sigma)\,.
\end{array} \right.$
\end{center} 
 
  \item III.  Continuous Initial-Dirichlet Problem: 
 \begin{center}
$(I\text{-}D)\left\{ \begin{array}{rl}Lu&=0  \textrm{ in } \Omega^T:=\Omega\cap \{t>T\}  \\
u(X,T)&=0  \text{ in } \Omega_{T}=  \Omega\cap \{t\equiv T\} \\
u\vert_{\Sigma^T}&=f\in C_c(\Sigma^T)
\\& u\in C(\Omega^T \cup \no^T)\,.
\end{array} \right.$
\end{center} 
Here, $\Sigma^T$ denotes the quasi-lateral boundary of the domain
$\Omega^T$.
The statement $u\vert_{\Sigma^T}=f\in C_c(\Sigma^T)$ is intepreted as in problem I,
and if  $\Omega^T$ is unbounded, 
we further specify that $\lim_{\|\bX\|\to\infty}
u(\bX)=0$.  
 
 \item IV. $L^p$ Initial-Dirichlet Problem: 
 \begin{center}
$(I\text{-}D)_p\left\{ \begin{array}{rl}Lu&=0  \textrm{ in } \Omega^T:=\Omega\cap \{t>T\}  \\
u(X,T)&=0  \text{ in }  \Omega_{T}=  \Omega\cap \{t\equiv T\}\\
u\vert_{\Sigma^T}&=f\in L^p(\Sigma^T)
\\&N_* u\in L^p(\Sigma^T)\,.
\end{array} \right.$
\end{center} 
\end{list}
In problems II and IV, the statement $u\vert_{\Sigma}=f\in L^p(\Sigma)$ (resp., 
$u\vert_{\Sigma^T}=f\in L^p(\Sigma^T)$) is understood in the sense of
parabolic non-tangential convergence.  We shall discuss this issue, as well as the precise definition of the non-tangential
maximal function $N_*u$, in the sequel. In problems III and IV, the statement
$u(X,T)=0  \text{ in } \Omega_{T}$ means that $u$ vanishes continuously on $\Omega_T$.
\end{definition}
 
\begin{definition}\label{parabolic}({\bf Caloric and Parabolic Measure})
Let $\Omega\subset \ree$ be an open set. 
Let $u$ be the  
PWB solution (see \cite{W1}, \cite[Chapter 8]{W2})
of the Dirichlet problem  for the heat equation, with data $f \in C_c(\eo)$. 
By the Perron construction, 
for each point $(X,t) \in \Omega$, 
the mapping 
$f \mapsto u(X,t)$ is  bounded, and by
the resolutivity of functions $f\in C(\eo)$ (see \cite[Theorem 8.26]{W2}),
it is also linear.
The caloric 
measure with pole $(X,t)$ is the probability measure $\omega^{X,t}$ given by 
the Riesz representation 
theorem, such that
\begin{equation}\label{parmeasuredef}
u(X,t)=\iint_{\eo}f(y,s)\,d\omega^{X,t}(y,s).
\end{equation}
For a general divergence form parabolic operator  $L$ as in \eqref{eq1.1}-\eqref{eq1.1*}, 
parabolic measure $\hm^{X,t}=\hm_L^{X,t}$ may be defined similarly, provided that  the continuous Dirichlet problem is solvable for $L$.
\end{definition}

\begin{definition}\label{defadr} ({\bf  ADR})  (aka {\it Ahlfors-David regular [in the parabolic sense]}).
Let $\Omega\subset \ree$. We say that  the quasi-lateral boundary 
$\Sigma$ 
 is {\em globally ADR} (or just ADR)
if  there is a constant $M_0$ such that for every parabolic cube $Q_r = Q_r(x,t)$, 
centered on $\Sigma$,
and corresponding surface cube
$\Delta_r=Q_r\cap \Sigma$, with $r<\diam(\Omega)$,
\begin{equation}\label{eq.adr}
\frac1{M_0} r^{n+1}\leq \sigma(\Delta_r) \leq M_0 r^{n+1}\,.
\end{equation}
 We also say that $\Sigma$ is ADR on a surface cube $\Delta=Q\cap \Sigma$, if there is
a constant $M_0$ such that \eqref{eq.adr} holds for every surface cube $\Delta_r=Q_r\cap\Sigma$,
with $Q_r \subset Q$ and centered on $\Sigma$.
\end{definition}

\begin{definition}\label{defbackadr} 
({\bf Time-Backwards ADR, aka TBADR})
Given a parabolic cube $Q$ centered on $\Sigma$, 
and corresponding surface cube
$\Delta=Q\cap \Sigma$,
we say that $\Sigma$ is time-backwards ADR on
$\Delta$ if it is ADR on $\Delta$, and if, in addition there exists a uniform constant $c_0>0$ 
such that
\begin{equation}\label{eq.backadr} 
c_0r^{n+1}\leq \sigma(\Delta^-_r)\,, 
\end{equation}
for every
$\Delta^-_r=Q^-_r\cap \Sigma,$ where
$Q_r\subset Q$ is centered at some 
point $(x,t)\in \Sigma$.   Note that by definition, 
if $\Sigma$ is TBADR on $\Delta = Q\cap\Sigma$, then it is TBADR on 
every $\Delta'=Q'\cap\Sigma$ with $Q'\subset Q$, and $Q'$ centered on $\Sigma$.

 If  $\Sigma$ is time-backwards ADR on every
$\Delta =\Sigma \cap Q_r(x_0,t_0)$, for all $(x_0,t_0)\in\Sigma$, and for all $r$ with 
$0<r<\sqrt{t_0-T_{min}}/(4\!\sqrt{n})$,  
then we shall simply say that $\Sigma$ is (globally) time-backwards ADR (and
we shall refer to such $\Delta$ as ``admissible"; note that if $T_{min} = -\infty$, then there is no restriction on $r$, and in that case every surface cube is admissible). 
\end{definition}

\begin{remark} \label{r1.22} The assumption of some backwards in time thickness, 
as in Definition \ref{defbackadr},
is rather typical in the parabolic setting.
 See, e.g., the backwards in time
 capacitary conditions in \cite{La}, \cite{EG}, \cite{GL}, \cite{FGL}, \cite{GZ}, \cite{BiM}. 
Moreover, it is not hard to verify that by the result of \cite{EG} (or of \cite{GL}, \cite{FGL}),  
time-backwards ADR on some surface cube $\Delta$ implies 
parabolic Wiener-type regularity of each point in $\Delta$ (and thus global time-backwards
ADR implies regularity of the parabolic boundary $\ppo$), in the case of the heat equation
\cite{EG}, or for $L$ with smooth coefficients \cite{GL}, or with $C^1$-Dini coefficients \cite{FGL}.
\end{remark}

 \begin{remark}\label{r-adr}
By \cite[Theorem 8.40]{W2}, the abnormal boundary $\ao$ is contained in a 
countable union of hyperplanes orthogonal
to the $t$-axis.   
Moreover, the same is true for the bottom boundary $\bo$, since its image under the change of variable
$t \to -t$ is contained in $\ao^*$, for the domain $\Omega^*$ obtained from $\Omega$ by the 
same change of variable.  Thus, $\sigma(\bo) = 0$.
 \end{remark}

\begin{remark}\label{r-adr2} The time-backwards ADR condition ensures that
the quasi-lateral boundary $\Sigma$ is a natural substitute for the lateral boundary, 
for the general class of domains that we consider;
in particular,  
$\sso =\emptyset=\so\setminus \{t\equiv T_{max}\}$, at least locally on any surface cube $\Delta$
on which TBADR holds,
and thus (except for the 
possible point at $\infty$),  $\eo =\ppo=\Sigma$,  in the set
$\{(X,t): t>T_{min}\}$.
 Moreover, if $\omega^{X,t}\ll\sigma$,  on some surface cube $\Delta$ (as we conclude in
 Theorem \ref{tmain}), then
$\hm^{X,t}(\bo \cap \Delta)=0$, by Remark \ref{r-adr}.
\end{remark}

\begin{remark}\label{1.8} Time-backwards ADR yields an 
apparently stronger property: specifically,  
 that if $\Sigma$ 
is time-backwards ADR on $\Delta=\Delta_r=\Sigma\cap Q_r(x_0,t_0)$, 
then \eqref{eq.backadr} self-improves to give the estimate
\begin{equation}\label{eq.backadrbetter} 
c_1r^{n+1}\leq \sigma(\Delta_r^-\cap \{t<t_0-(ar)^2\})\,, 
\end{equation}
for some constants $a\in(0,1)$ and $c_1>0$, depending only on $n$ and the ADR and TBADR constants;
see \cite[Apprendix A]{GH} for the proof.
\end{remark}

\begin{definition}({\bf Parabolic BMO}). \label{bmodef}
$BMO(\Sigma)$ is the parabolic version of the usual BMO space with norm $||f||_{BMO(\Sigma)}$, 
defined for any locally integrable function $f$  on $\Sigma$ by
\begin{equation}
||f||_{BMO(\Sigma)}:=\underset{\Delta}\sup\Big\{ \displaystyle\tiltfiint_{\Delta} |f-f_{\Delta}|\,d\sigma\Big\}<\infty,
\end{equation}
where $\Delta = \Delta_r(x,t) :=Q_r(x,t)\cap\Sigma$, $ f_\Delta:= \tiltfiint_\Delta f$, $(x,t)\in\Sigma$, and $0<r<R_0$.
\end{definition}

\begin{definition}({\bf Parabolic Polar Coordinates}).\label{polar} Let 
$d\sigma_{\mathbb{S}^n}$ denote the usual surface measure on the unit sphere
$\mathbb{S}^n$ in $\ree$.
We have the parabolic polar coordinate decomposition 
\[
(X,t)=(\rho\zeta,\rho^2 \tau)\, ,\quad dXdt= \rho ^{n+1}d\rho \, 
d\mu(\zeta,\tau) , 
\]
where $(\zeta,\tau) \in \mathbb{S}^{n}$,  $\rho
= \|(X,t)\|$, and $\mu$ is an appropriately weighted version of surface measure on
the sphere; to be precise, $d\mu(\zeta,\tau):= \big(1+\tau^2\big)\, d\sigma_{\mathbb{S}^n}(\zeta,\tau)$;
see, e.g. \cite{FR1,FR2} or \cite{R}.

\end{definition}

\begin{definition} \label{projdef} ({\bf Parabolic Projection}).
We denote by $\pi_{\text{par}}(X,t)$ the parabolic projection of $(X,t)$ onto $\mathbb{S}^n$, 
which we define by 
setting $\pi_{\text{par}}(X,t)=(\zeta,\tau)$, where $(X,t)$ has the parabolic polar coordinate 
representation
\begin{align*}
    (X,t)=\left(\rho\zeta, \rho^2\tau\right),
\end{align*}
with $\rho=||(X,t)||$, and $(\zeta,\tau)\in \mathbb{S}^n$.
\end{definition}
\begin{definition} \label{conedef} ({\bf Parabolic Cone})  Let 
$(\zeta,\tau)\in \mathbb{S}^n$, and let $\vartheta>0$.  We define the
parabolic cone $\Gamma_{\vartheta}(\zeta,\tau)$, ``in the direction $(\zeta,\tau)$",
with vertex  at the origin and 
aperture $\vartheta>0$, as follows:
\begin{equation*}
\Gamma_{\vartheta}(\zeta,\tau):=\{(Y,s) : \|\pi_{\text{par}}(Y,s)-(\zeta,\tau)\|<\vartheta\}.
\end{equation*}
For any $(X,t)\in\ree$ with $\pi_{\text par}(X,t) =(\zeta,\tau)$, we shall also write
\[\Gamma_{\vartheta}(X,t):= \Gamma_{\vartheta}(\zeta,\tau)\,.\] 
\end{definition}

\begin{definition}\label{defAinfty}
({\bf $A_\infty$}, weak-$A_\infty$, and weak-$RH_q$). 
Given a parabolic ADR set $E\subset\ree$, 
and a surface cube
$\Delta_0:= Q_0 \cap E$,
we say that a Borel measure $\mu$ defined on $E$ belongs to
$A_\infty(\Delta_0)$ if there are positive constants $C$ and $\theta$
such that for each surface cube $\Delta = Q\cap E$, with $Q\subseteq Q_0$,
we have
\begin{equation}\label{eq1.ainfty}
\mu (F) \leq C \left(\frac{\sigma(F)}{\sigma(\Delta)}\right)^\theta\,\mu (\Delta)\,,
\qquad \mbox{for every Borel set } F\subset \Delta\,.
\end{equation}
Similarly, we say that $\mu \in$ weak-$A_\infty(\Delta_0)$ if 
for each surface cube $\Delta = Q\cap E$, with $2Q\subseteq Q_0$,
\begin{equation}\label{eq1.wainfty}
\mu (F) \leq C \left(\frac{\sigma(F)}{\sigma(\Delta)}\right)^\theta\,\mu (2\Delta)\,,
\qquad \mbox{for every Borel set } F\subset \Delta\,.
\end{equation}
We recall that, as is well known, the condition $\mu \in$ weak-$A_\infty(\Delta_0)$
is equivalent to the property that $\mu \ll \sigma$ in $\Delta_0$, and that for some $q>1$, the
Radon-Nikodym derivative $k:= d\mu/d\sigma$ satisfies
the weak reverse H\"older estimate 
\begin{equation}\label{eq1.wRH}
\left(\,\, \tiltfiint_\Delta k^q d\sigma \right)^{1/q} \,\leq C\, \tiltfiint_{2\Delta} k \,d\sigma\, 
\approx\,  \frac{\mu(2\Delta)}{\sigma(\Delta)}\,,
\quad \forall\, \Delta = Q\cap E,\,\, {\rm with} \,\, 2Q\subseteq Q_0\,.
\end{equation}
We shall refer to the inequality in \eqref{eq1.wRH} as
an  ``$RH_q$" estimate, and we shall say that $k\in RH_q(\Delta_0)$ if $k$ satisfies \eqref{eq1.wRH}.
\end{definition}

\section{Preliminaries}\label{s2}
The proofs of the following two lemmas 
may be found in the Appendix of \cite{GH}.

Let $a>0$ be the constant mentioned in Remark \ref{1.8}.  In the sequel, $\Omega$ will always denote
an open set in $\ree$, with quasi-lateral boundary $\Sigma$.
To simplify terminology, in the sequel we shall say that some quantity ``depends on
ADR" if it depends on the constants in the ADR and/or time backwards ADR conditions.
We recall that  $\hm^{X,t}$ may denote either caloric measure, 
or parabolic measure for
a divergence form parabolic operator  as in \eqref{eq1.1}-\eqref{eq1.1*}, 
but in the latter case we implicitly assume that the continuous
Dirichlet problem is solvable for $L$;  as mentioned above (see Remark \ref{r1.22}),
given our time-backwards ADR assumption, such solvability indeed holds for the heat equation,
and more generally for equations with $C^1$-Dini coefficients, by the result of
\cite{FGL}.  Recall that $R_0:=\diam(\Sigma)$.

\begin{lemma}[Parabolic Bourgain-type Estimate]\label{Bourgain} 
Let $\Sigma$ be time-backwards ADR on $ \Delta_r := Q_r(x_0,t_0) \cap \Sigma$,
where $(x_0,t_0)\in\Sigma$, 
and $0<r< \min\big(R_0,\sqrt{t_0-T_{min}}/\big(4\!\sqrt{n}\,\big)\big)$. 
Then there exists $M_1,\kappa>0$ such that for all $(X,t)\in Q_{\frac{a}{M_1}r}\cap \Omega$, 
\begin{equation}\label{eq2.3}
\omega^{X,t}(\Delta_{r})\geq \kappa\,,
\end{equation}
where $Q_{\frac{a}{M_1}r}:=Q\big((x_0,t_0),\frac{a}{M_1}r\big)$.  The constants
$M_1$ and $\kappa$ depend only on 
$n$, ADR and $\lambda$.
\end{lemma}

\begin{remark}\label{2.4} One may readily deduce the following
consequence of Lemma \ref{Bourgain}.  Let $\Sigma$ be globally TBADR. 
Then there is a constant $M_2\approx_n M_1/a$,  such that, given
$(X,t)\in \Omega$, with $2M_2\delta_\infty(X,t) <\min\big(R_0,\sqrt{t-T_{min}}\,\big)\big)$,
if $(\hat{x},\hat{t})\in\Sigma$ is a touching point for $(X,t)$, so that 
$||(X,t)-(\hat{x},\hat{t})||_{\ell^\infty}=\delta_\infty(X,t)=:r$, and if
\begin{equation}\label{dxtdef}
\Delta_{X,t}:=\Delta\big((\hat{x},\hat{t}),M_2r\big)=\Sigma\cap Q\left((\hat{x},\hat{t}), M_2r\right)\,,
\end{equation} 
then 
\begin{equation}\label{2.5}
\omega^{X,t}(\Delta_{X,t})\geq \kappa.
\end{equation}
\end{remark}

\begin{lemma}[H{\"o}lder Continuity at the Boundary]\label{H}
Let $(x_0,t_0)\in \Sigma$, and fix $r$ with
$0<r< \min\big(R_0,\sqrt{t_0-T_{min}}/\big(8\!\sqrt{n}\,\big)\big)$.  
Suppose that 
$\Sigma$ is time-backwards ADR on $ \Delta_{2r} := Q_r(x_0,t_0) \cap \Sigma$.  Let $u$ be the
parabolic 
measure solution corresponding to non-negative data $f\in C_c(\eo)$, with 
$f\equiv 0$ on $\Delta_{2r}$. 
Then for some $\alpha>0$,
\begin{align*}
u(Y,t)\leq C\left(\dfrac{\delta(Y,t)}{r}\right)^{\alpha}\dfrac{1}{| Q_{2r}(x_0,t_0))|}\iint_{ Q_{2r}(x_0,t_0)\cap \Omega} u, \hspace{.1in} \forall (Y,t)\in Q_r(x_0,t_0)\cap \Omega, 
\end{align*}
where the constants $C$ and $\alpha$ depend only on $n,$ $\lambda$, and 
the ADR and time-backwards ADR constants.
\end{lemma}

\section{Proof of Theorem \ref{tmain}}\label{s3}

Recall that for
$(X,t)\in \Omega$, we let $\delta_\infty(X,t):=\text{dist}_\infty((X,t), \eo)$ denote the 
parabolic $\ell^\infty$ distance to the essential boundary, and
that if $\sqrt{t-T_{min}} > \delta_\infty(X,t)$, then 
\begin{equation}\label{eq3.1}
\delta_\infty(X,t)=\text{dist}_\infty((X,t),\Sigma)\,,
\end{equation}
by definition of $\Sigma$.  We note that in the context of Theorem \ref{tmain}, by hypothesis
we shall always
work with points $(X,t)$ for which \eqref{eq3.1} holds.



Given $(X,t)\in\Omega$, let  $(\hat{x},\hat{t})\in \Sigma$ be a touching point for $(X,t)$, so that
\begin{equation}\label{eq3.4}
r:=\delta_\infty(X,t)=||(X,t)-(\hat{x},\hat{t})||_{\ell^\infty}\,,
\end{equation}
 and define $\Delta_{X,t}$ as in \eqref{dxtdef}
where $M_2$ is the constant in Remark \ref{2.4}.
We shall say that caloric (or parabolic) measure $\omega^{X,t}$ is {\em locally ample} on $\Delta_{X,t}$,
or more precisely,
{\em $(\theta,\beta)$-locally ample},  if
there exists  constants $\theta,\,\beta\in (0,1)$ such that 
\begin{align}\label{eq1.4}
\sigma(F)\geq 
(1-\theta)\sigma(\Delta_{X,t})\,\,\implies\,\,
\omega^{X,t}(F)= \omega_L^{X,t}(F)\geq\beta\,,
\end{align}
where $F\subset \Delta_{X,t}$ is a Borel set.

We shall use the following result from \cite{GH}; we remark that it is the parabolic analogue of a result proved
in the elliptic setting in \cite{BL}.
\begin{theorem}\cite[Theorem 1.6]{GH}.
\label{lemma} Let $\Omega\subset \ree$ be 
an open set with a globally ADR
quasi-lateral boundary $\Sigma$. 
Let $(x_0,t_0)\in\Sigma$, and let $0<r<\sqrt{t_0-T_{min}}/(8\!\sqrt{n})$. 
Assume that 
$\Sigma$ is time-backwards ADR 
on $ \Delta_{2r} =\Sigma \cap Q_{2r}(x_0,t_0)$, 
and suppose that there are constants $\theta,\,\beta\in (0,1)$ such that
caloric measure $\omega^{X,t}$ satisfies the $(\theta,\beta)$-local ampleness condition 
\eqref{eq1.4} on $\Delta_{X,t}$ for each
$(X,t) \in \Omega \cap Q_{2r}(x_0,t_0)$. 

Then there exist constants $C\geq 1$, $\gamma >0$,  
such that if $(Y_0,s_0) \in \Omega \setminus Q_{4r}(x_0,t_0)$,
then $\omega^{Y_0,s_0} \ll\sigma$ on $\Sigma \cap Q_r(x_0,t_0)$, with
$d\omega^{Y_0,s_0}/d\sigma=h$ satisfying
\begin{multline}\label{RH}
\left(\rho^{-n-1}\iint_{\Delta_{\rho}(y,s)}h^{1+\gamma}d\sigma\right)^{1/(1+\gamma)} \leq 
C\rho^{-n-1}
\iint_{\Delta_{2\rho}(y,s)}h\, d\sigma 
\\=C\rho^{-n-1}\omega^{Y_0,s_0}\left(\Delta_{2\rho}(y,s)\right),
\end{multline}
whenever $(y,s)\in\Sigma$ and $Q_{2\rho}(y,s)\subset Q_{r}(x_0,t_0)$, where 
$\Delta_{\rho}(y,s)=Q_{\rho}(y,s)\cap \Sigma$, and 
$\Delta_{2\rho}(y,s)=Q_{2\rho}(y,s)\cap \Sigma$.
\end{theorem}
\begin{remark}
In \cite{GH}, $\Delta_{X,t}$ is defined in a slightly different way: 
there, $\Delta_{X,t}$ is centered at $(X,t)\in \Omega$; more precisely, it is
of the form $\Sigma \cap Q\big((X,t),K\delta(X,t)\big)$, for some $K\geq 2$.
This is comparable to the present definition of $\Delta_{X,t}$ in 
Remark \ref{2.4}. 
\end{remark}

Thus, to prove Theorem \ref{tmain}, we suppose that
 $\Sigma$ is globally ADR and TBADR, and observe that 
 it suffices to verify the hypotheses of Theorem \ref{lemma}, 
in the presence of BMO-solvability. More precisely, we suppose 
that estimate \eqref{1.3} holds for all 
$f\in C_c(\Sigma\cap\{T_{min}<t<T_{max}\})$, and 
our goal is to verify the 
$(\theta,\beta)$-locally ampleness condition \eqref{eq1.4}, for all $(X,t)\in \Omega$ with
$2M_2\delta_\infty(X,t) <\min\big(R_0,\sqrt{t-T_{min}}\,\big)\big)$, where $M_2$ is the constant in Remark \ref{2.4}.  
In comparing this constraint on $\delta_\infty(X,t)$ with that on $r$ in Theorem \ref{lemma}, we observe that
there is no loss of generality: indeed,
for a fixed large constant $M$,  we may cover a given surface cube
$\Delta_r(x_0,t_0)$ by surface cubes of scale $r/M$; it then suffices to verify 
the Reverse H\"older inequality
\eqref{RH} on these smaller cubes.

We now fix $(X,t)$ as above, and let $(\hat{x},\hat{t})\in \Sigma$ be a touching point for $(X,t)$, so that 
\eqref{eq3.4} holds.  
Fix a sufficiently small number 
$b\in (0,\pi/10,000)$, to be chosen depending only on $n$ and ADR. 
We then set 
$$Q_{X,t}:=Q((\hat{x},\hat{t}),M_2r), \quad \Delta_{X,t}:=\Delta((\hat{x},\hat{t}),M_2r),$$
$$Q_{X,t}':=Q((\hat{x},\hat{t}),br), \quad \Delta_{X,t}':=\Delta((\hat{x},\hat{t}),br).$$
Note that  $\Delta_{X,t}$ is the same as in \eqref{dxtdef}.  

The proof will use the following pair of claims.  We recall that
$a$ is the
constant in Remark \ref{1.8}. 

\medskip

\noindent\textbf{Claim 1}: For $b$ small enough, depending on $n, a$ and ADR, there is a 
constant $\beta>0$ depending only on $n,a,b$, ADR and $\lambda$, and a cube 
$Q_1:=Q\big((x_1,t_1),br\big)\subset Q_{X,t}$, with $(x_1,t_1)\in \Sigma$, 
such that 
\begin{equation}\label{sep}
\dist(Q_{X,t}',Q_1)\gtrsim_a r
\end{equation}
(note that the implicit constants in \eqref{sep} depend on the constant $a$ in
Remark \ref{r1.22}, but {\em not} on $b$), and
\begin{equation}\label{3.3}
\omega^{X,t}(\Delta_1)\geq\beta\omega^{X,t}(\Delta_{X,t}),
\end{equation}
where $\Delta_1:=Q_1\cap \Sigma$.

\begin{remark} Since the constant $a$ in Remark \ref{1.8}
depends only on $n$ and ADR, in turn $b$ ultimately
depends only on $n$ and ADR.
\end{remark}

\smallskip

\noindent
\textbf{Claim 2}: Suppose that $u$ is a non-negative solution of $Lu=0$ in $\Omega$, vanishing continuously on $2\Delta_{X,t}'$, with $||u||_{L^{\infty}(\Omega)}\leq 1$. Then for every $\epsilon>0$,
\begin{equation}\label{3.4}
u^2(X,t)
\leq\dfrac{ C_{\epsilon}}{\sigma(\Delta_{X,t})}\displaystyle\iint_{Q_{X,t}\cap \Omega}\!
\left( |\nabla u(Y,s)|^2+|\delta(Y,s)\partial_s u(Y,s)|^2\right)\delta(Y,s)dYds+C\epsilon^{2\alpha},
\end{equation}
where $\alpha$ is the exponent from Lemma \ref{H}.

Momentarily taking these two claims for granted, we adapt to the parabolic setting 
the argument of \cite{DKP}, as modified in
\cite{HLe}. 
Let $Q_1$ and $\Delta_1$ be as in Claim 1. Let $F\subset \Delta_{X,t}$ be a Borel set satisfying
\begin{align*}
\sigma(F)\geq (1-\eta)\sigma(\Delta_{X,t}),
\end{align*}
for some small $\eta>0$. If we choose $\eta$ small enough, depending only on $n$, ADR, and $b$, then by inner regularity of $\sigma$, there is a closed set $F_1\subset F\cap \Delta_1$ such that
\begin{align*}
\sigma(F_1)\geq (1-\sqrt{\eta})\sigma(\Delta_1).
\end{align*}
Set $A_1:=\Delta_1\setminus F_1$. Then $A_1$ is relatively open in $\Sigma$. Define 
\begin{align*}
f:=\max(0,1+\gamma \log \mathcal{M}(1_{A_1})),
\end{align*}
where $\gamma>0$ is a small number, to be chosen, and $\mathcal{M}$ is the Hardy-Littlewood maximal operator on $\Sigma$.
Note that we have the following:
\begin{align}\label{est}
0\leq f\leq 1, \qquad ||f||_{BMO(\Sigma)}\leq C\gamma ,\qquad 1_{A_1}(\bx)\leq f(\bx),\, \,\forall \bx\in\Sigma.
\end{align}
Note also that if ${\bf z}\in \Sigma\setminus 2Q_1$, then
\begin{center}
$\mathcal{M}(1_{A_1})({\bf z})\lesssim \dfrac{\sigma(A_1)}{\sigma(\Delta_1)}\lesssim \sqrt{\eta},$
\end{center}
where the implicit constants depend only on $n$ and ADR. Thus, if $\eta$ is chosen small enough depending on $\gamma$, then $1+\gamma\log\mathcal{M}(1_{A_1})$ will be negative, hence $f\equiv 0$, on $\Sigma\setminus 2Q_1$.

In order to work with continuous data, we shall require the following.
\begin{lemma}\label{3.5} There exists a collection of continuous functions $\{f_\nu\}_{0<\nu<ar/1000},$ 
defined on $\Sigma$ with the following properties.
\begin{enumerate}
\item $0\leq f_\nu\leq 1$, for each $\nu$.
\item supp$(f_\nu)\subset 3Q_1\cap \Sigma.$
\item $1_{A_1}(\bx)\leq \underset{\nu\to 0}\liminf f_\nu(\bx)$, for every $\bx\in \Sigma$.
\item $\underset{\nu}\sup\|f_\nu\|_{BMO(\Sigma)}\leq C\|f\|_{BMO(\Sigma)}\lesssim \gamma,$ 
where $C=C(n,\text{ADR})$.
\end{enumerate}
\end{lemma}
We defer the proof of Lemma \ref{3.5} to the end of this section.

Taking the two claims (and Lemma \ref{3.5})
for granted momentarily, we give the proof of Theorem \ref{tmain}.  As noted above, by Theorem \ref{lemma},
it suffices to verify the $(\theta,\beta)$-locally ampleness condition \eqref{eq1.4}.  To this end,
let $u_\nu$ be the solution of the continuous Dirichlet problem with data $f_\nu$. 
Then $f_\nu$ vanishes on $2\Delta_{X,t}'$, by the separation condition \eqref{sep} 
in Claim 1 and Lemma \ref{3.5}-(2), provided that $b$ is chosen small enough depending on $a$. Then, for small $\epsilon>0$ to be chosen momentarily, by Lemma \ref{3.5}, Fatou's lemma, and Claim 2, we have
\begin{equation}\label{3.6}
\omega^{X,t}(A_1)\leq \int_{\Sigma} \underset{\nu\to 0}\liminf f_\nu \,d\omega^{X,t} \leq \underset{\nu\to 0}\liminf \,u_\nu(X,t) \leq C_{\epsilon}\gamma + C\epsilon^{\alpha},
\end{equation}
where in the last inequality we used \eqref{3.4}, \eqref{1.3}, and Lemma \ref{3.5}-(4).
Combining \eqref{3.6} with \eqref{2.5}, we find that
\begin{equation}\label{3.7}
\omega^{X,t}(A_1) \leq (C_{\epsilon}\gamma + C\epsilon^{\alpha})\omega^{X,t}(\Delta_{X,t}).
\end{equation}

Next, we set $A:=\Delta_{X,t}\setminus F$, and observe that by definition of $A$ and $A_1$, along with Claim 1, and \eqref{3.7},
\begin{center}
$\omega^{X,t}(A)\leq \omega^{X,t}(\Delta_{X,t}\setminus \Delta_1)+\omega^{X,t}(A_1)\leq (1-\beta+C_{\epsilon}\gamma + C\epsilon^{\alpha})\omega^{X,t}(\Delta_{X,t}).$
\end{center}
We now choose first $\epsilon >0$, and then $\gamma>0$, so that $C_{\epsilon}\gamma + C\epsilon^{\alpha}<\beta/2$, to obtain that
\begin{center}
$\omega^{X,t}(F)\geq \dfrac{\beta}{2}\omega^{X,t}(\Delta_{X,t})\geq c\beta,$
\end{center}
where in the last inequality we have used \eqref{2.5}. Therefore \eqref{eq1.4} holds.

It remains to prove the two claims. Let $a>0$ be the constant mentioned in Remark \ref{1.8}. Recall 
that $M_1$ is the constant in Lemma \ref{Bourgain}, and that $M_2$ is the constant in Remark \ref{2.4}.

\begin{proof}[Proof of Claim 1]  Recall that we have fixed
$(X,t)\in \Omega$, and that $(\hat{x},\hat{t})\in \Sigma$ is a touching point for $(X,t)$, so that
$(\hat{x},\hat{t})$ lies on the boundary of 
the (open) cube $Q_r(X,t)$, with $r=\delta_\infty(X,t) = ||(X,t)-(\hat{x},\hat{t})||_{\ell^\infty}$, and
 $Q_r(X,t) \cap \Sigma = \emptyset$. If there is more than one touching point, we simply fix one.
 Note that since $(\hat{x},\hat{t})\in\partial Q_r(X,t)$, we have in particular that
 \[ \hat{t} \leq t+ r^2\,.\]
Consequently, we may apply Remark \ref{1.8} to the cube
$Q_{big}:= Q_{2a^{-1} r}(\hat{x},\hat{t})$, to find a point
$(y,s) \in \Sigma \cap Q_{big}$, with $s< \hat{t} - (2r)^2\leq t +r^2 -(2r)^2$.
The point $(y,s)$ therefore satisfies
\begin{equation}\label{eq3.14}
s< t-3r^2\,,\,\, \text{and}\,\, \, \|(X,t) - (y,s)\| \lesssim_a r\,.
\end{equation}

Let us note for future reference that for
$(Z,\tau)\in \Omega \cap Q_{big}$, 
by Remark \ref{r-adr2} we have 
\begin{equation}\label{eqdist2}
\dist_\infty\big((Z,\tau),\Sigma\big)= \delta_\infty(Z,\tau) 
\approx_a \dist_\infty\big((Z,\tau),\po\big)\,, \,\, \text{if }\, \tau<t-(r/4)^2\,,
\end{equation}
since $(X,t) \in \Omega$ implies that $t < T_{max}$, and the restriction
$\sqrt{t-T_{min}}>2M_2 r$, with $M_2 \approx M_1/a\gg 1/a$, 
implies that \eqref{eq3.1} holds for $(Z,\tau)\in Q_{big}$.

We fix a point ${\bf X}_*=(X_*,t_*)$ lying on the back face of $Q_r(X,t)$ (so that $t_* = t-r^2$),
with 
\begin{equation}\label{eq3.15}
|X_* -\hat{x}| \geq r/4\,. 
\end{equation}
We now form the parabola $\p_1$ with vertex at $(y,s)$, passing through the point $(X_*,t_*)$, 
so that any point 
$(Z,\tau)$ on $\p_1$ satisfies 
\[\tau-s = \,\frac{t_* -s}{|X_*-y|^2} \,|Z-y|^2\, \gtrsim_a\, |Z-y|^2\,.\]
We also form the parabola $\p_2$, with vertex at $(X_*,t_*)$, through the point $(X,t))$, so that any
point $(Z,\tau)$ on $\p_2$ satisfies 
\[\tau-t_* = \,\frac{t-t_* }{|X-X_*|^2} \,|Z-X_*|^2\, \gtrsim\, |Z-X_*|^2\]
(it may be that $X_*=X$, in which case $\p_2$ is simply the horizontal line joining
$(X,t-r^2)$ to $(X,t)$).
Set  $\C:= \p_1\cup\p_2$, and travel along $\C$ backwards in time, starting at $(X,t)$, moving towards 
$(X_*,t_*)$, and if need be through $(X_*,t_*)$ towards $(y,s)$,  stopping the 
first time that we reach a point $(Z_1,\tau_1)$ satisfying
\[\delta_\infty(Z_1,\tau_1) \, = \, b M_2^{-1} r\,.\]
Choose $(x_1,t_1)\in \Sigma$ such that $\delta_\infty(Z_1,\tau_1)= \|(Z_1,\tau_1) -(x_1,t_1)\|_{\ell^\infty}$,
set $\Delta_1:= Q_1\cap\Sigma$, with $Q_1:= Q\big((x_1,t_1), br\big)$,  
so that, by Remark \ref{2.4}, 
\[\hm^{Z_1,\tau_1}(\Delta_1)\geq \kappa\,.\]
We may then move along $\C$, forwards in time, from $(Z_1,\tau_1)$ to $(X,t)$, to
obtain \eqref{3.3} by Harnack's inequality and
\eqref{eqdist2}, and the fact that $\hm^{X,t}$ is a probability measure.

Moreover, by \eqref{eq3.15} and the construction of the curve $\C$, for $b$ small enough  
depending on $a$, 
we readily obtain the separation condition \eqref{sep}, 
and for $M_2$ large enough, again depending on $a$, using the second inequality in \eqref{eq3.14},
we obtain
the containment $Q_1\subset Q_{X,t}$. 
\end{proof}

\begin{proof}[Proof of Claim 2]

\noindent By a translation, we may suppose that the touching point $(\hat{x},\hat{t})$ is the origin.
As above, we set 
\[r := \delta_\infty(X,t) = \|(X,t)\|_{\ell^\infty},\] 
where we have used that
$(\hat{x},\hat{t}) = 0$.  Since the $\ell^2$ and $\ell^\infty$ versions of the parabolic distance
are comparable, we have that
\begin{equation}\label{eq3.15aa}
r_1:=\|(X,t)\|\approx \delta(X,t)\approx r\,,
\end{equation} with implicit constants depending only on dimension. 

Set
\begin{equation}\label{eq3.16}
P_{X,t} := Q_{cr}(X,t)\,,\quad P_{X,t}^- := Q^-_{cr}(X,t)\,,
\end{equation}
where $c<1/1000$ is a small fixed positive constant to be chosen momentarily.  Then by
\cite[Theorem 3]{M},  we have that
\begin{equation}\label{eq3.17}
u(X,t) \lesssim \left(\tiltfiint_{P^-_{X,t}} |u({\bf Y})|^2 d{\bf Y}\right)^{1/2}\,.
\end{equation}

Let $S_0$ denote the spherical cap
\[S_0:= \left\{(\zeta,\tau) \in \mathbb{S}^n:\, \|(\zeta,\tau)-\pi_{\text par}(X,t) \|< \pi/1000\right\}.\]
For the sake of notational convenience, we shall write
\[ \xi=(\zeta,\tau)\,,\quad \rho^{(1,2)} \xi := (\rho\zeta,\rho^2\tau)\]
to denote, respectively, points on the unit sphere $\mathbb{S}^n$, and on the parabolic sphere of
radius $\rho$ (expressed in parabolic polar coordinates; see Definition \ref{polar}).

Then for $c$ in \eqref{eq3.16}  chosen small enough, we have that 
\[ P_{X,t}^- \subset \A_{X,t}\,,\]
where $\A_{X,t}$ is the region given in parabolic polar coordinates by
\[\A_{X,t} := \left\{ \rho^{(1,2)} \xi:\, \xi \in S_0,\, r_1/2 <\rho <R(\xi)\right\}\,,
\]
where $r_1 \approx r$ is defined in \eqref{eq3.15aa}, and $R(\xi)$ is defined appropriately so that 
$R(\xi)\lesssim r$, uniformly in $\xi$, and so that  $\A_{X,t} \subset \Omega$\footnote{
We need be this careful only if $T_{max}-t \lesssim r^2$, otherwise,
we could simply set $R(\xi) = Cr$}.  In fact, more generally, 
\begin{equation}\label{eq3.18}
\Gamma_{X,t} :=  \left\{ \rho^{(1,2)} \xi:\, \xi \in S_0,\, 0 <\rho <R(\xi)\right\}\subset \Omega\,.
\end{equation}
Of course, $\Gamma_{X,t}$ is just a truncated version of
 the parabolic cone  $\Gamma$ (see Definition \ref{conedef}) 
with vertex at $0=(\hat{x},\hat{t})$, 
in the direction $\pi_{\text par}(X,t)$, with aperture $\pi/1000$. 

Then by \eqref{eq3.17} and the fact that $P_{X,t}^- \subset \A_{X,t}$, we have
\begin{align*}
u(X,t) & \lesssim 
\left(r^{-n-2} \int_{S_0}\int_{r_1/2}^{R(\xi)} 
\left| u\big(\rho^{(1,2)}\xi\big)\right|^2 \, \rho^{n+1}d\rho \,d\mu(\xi)\right)^{1/2}
\\[4pt]&\lesssim \left(r^{-n-2} \int_{S_0}\int_{r_1/2}^{R(\xi)} 
\left| u\big(\rho^{(1,2)}\xi\big)- u\big((\epsilon r)^{(1,2)}\xi\big)\right|^2
 \, \rho^{n+1}d\rho \,d\mu(\xi)\right)^{1/2} + \, O(\epsilon^\alpha)
\\[4pt]& =:\, I \,+\, O(\epsilon^\alpha)
\end{align*}
where we have used 
parabolic polar coordinates  (Definition \ref{polar}), 
and where the ``big-O" term
\[ \left(r^{-n-2} \int_{S_0}\int_{r_1/2}^{R(\xi)} 
\left| u\big((\epsilon r)^{(1,2)}\xi\big)\right|^2
 \, \rho^{n+1}d\rho \,d\mu(\xi)\right)^{1/2} \approx
\left(\int_{S_0} \left|u\big((\epsilon r)^{(1,2)}\xi\big)\right|^2\,d\mu(\xi)\right)^{1/2}\]
has been estimated using first that
$r_1 \approx r\approx R(\xi)$, and then Lemma \ref{H} and the fact that 
$u$ vanishes continuously on $2\Delta_{X,t}'$, which is centered at $(\hat{x},\hat{t}) =0$, and has 
parabolic diameter $\approx r$.

It remains to control term $I$ by appropriate localized square functions.  To this end, 
using that $\rho \approx r$ in $\A_{X,t}$, we write
\begin{align*}
 I^2 &=r^{-n-2}\iint_{\A_{X,t}}\Big|\int_{\epsilon r}^{\rho} 
\partial_{q} u(q^{(1,2)}\xi)dq \Big|^2 \rho^{n+1}d\rho d\mu(\xi)
\\[4pt]&\lesssim \iint_{\A_{X,t}} \int_{\epsilon r}^\rho |\nabla u(q^{(1,2)} \xi)|^2 dq d\rho d\mu(\xi) 
+ \iint_{\A_{X,t}} \int_{\epsilon r}^\rho q^2 |\partial_s u(q^{(1,2)} \xi)|^2 dqd\rho d\mu(\xi)\\[4pt]
&:=I^2_1+I^2_2,
\end{align*}
We note first that
\begin{multline*} I_1^2
 \lesssim_\epsilon r^{-n-2}\iint_{\A_{X,t}} \int_{\epsilon r}^\rho 
|\nabla u(q^{(1,2)} \xi)|^2  q^{n+2} dq d\rho d\mu(\xi) \\[4pt]
\lesssim_\epsilon   r^{-n-1}\int_{S_0} \int_{\epsilon r}^{R(\xi)}
|\nabla u(q^{(1,2)} \xi)|^2  q^{n+2} dq  d\mu(\xi)
\\[4pt]\approx_\epsilon r^{-n-1}
 \iint_{\A_*} \big|\nabla u(Y,s)\big|^2 \delta(Y,s) dY\, ds
 \\[4pt] \lesssim_\epsilon \sigma\big(\Delta_{X,t}\big)^{-1}
  \iint_{Q_{X,t}\cap\Omega} \big|\nabla u(Y,s)\big|^2 \delta(Y,s) dY\, ds\,,
\end{multline*}
where the region $\A_*$ is given in parabolic polar co-ordinates by
\[\A_* := \left\{ q^{(1,2)} \xi:\, \xi \in S_0,\, \epsilon r <q <R(\xi)\right\}\,,
\]
and where in the last step we have used  \eqref{eq3.15aa},  \eqref{eq3.18}, ADR,
and the definitions
of $Q_{X,t}$ and $\Delta_{X,t}$.

Similarly,
\begin{multline*} I_2^2 \lesssim_\epsilon
r^{-n-2} \iint_{\A_{X,t}} \int_{\epsilon r}^\rho q^{n+4} |\partial_s u(q^{(1,2)} \xi)|^2 dqd\rho d\mu(\xi)
\\[4pt]
\lesssim_\epsilon   r^{-n-1}\int_{S_0} \int_{\epsilon r}^{R(\xi)}
|\partial_s u(q^{(1,2)} \xi)|^2  q^{n+4} dq  d\mu(\xi)
\\[4pt]
\approx_\epsilon r^{-n-1}
 \iint_{\A_*} \big|\partial_s u(Y,s)\big|^2 \delta^3(Y,s) dY\, ds
 \\[4pt]
\lesssim_\epsilon  \sigma\big(\Delta_{X,t}\big)^{-1}
 \iint_{\A_*} \big|\partial_s u(Y,s)\big|^2 \delta^3(Y,s) dY\, ds\,.
\end{multline*}

This concludes the proof of Claim 2, and hence of Theorem \ref{tmain}, modulo the proof of Lemma \ref{3.5}.
\end{proof}

\begin{proof}[Proof of Lemma \ref{3.5}]
Let $\zeta\in C_0^{\infty}(\ree),$   
$$\supp(\zeta)\subset B(0,1)\,,\quad \zeta \equiv 1 \,\, {\rm on}\,\, B(0,1/2)\,,\quad 0\leq \zeta\leq 1\,.  $$
Given $\nu\in (0,ar/1000)$, and $\bx, \bz:=(z,\tau) \in \Sigma$, set   
$$\Lambda_{\nu}(\bx,\bz):= b(\bx,\nu)^{-1} \zeta\left(\dfrac{\bx-\bz}{\nu^{\alpha}}\right)\,,$$
for $\alpha=(1,1,...,1,2)$, and
\begin{equation}\label{eq3.11}
b(\bx,\nu):= \iint_{\Sigma} \zeta\left(\dfrac{\bx-\bz}{\nu^{\alpha}}\right)\, d\sigma(\bz)\,\approx \nu^{n+1}\,,
\end{equation}
uniformly in $\bx\in \Sigma$, by the ADR property. Furthermore,
\begin{align*}
    \iint_{\Sigma} \Lambda_{\nu}(\bx,\bz)d\sigma(\bz)\equiv 1, \quad \forall \bx\in \Sigma.
\end{align*}
We now define
\begin{align*}
    f_\nu(\bx):=\iint_{\Sigma} \Lambda_{\nu}(\bx,\bz) f(\bz)d\sigma(\bz),
\end{align*}
so that $f_\nu$ is continuous, by construction. Let us now verify (1)-(4) of
Lemma \ref{3.5}. We obtain (1) immediately, by \eqref{est}, and the properties of $\Lambda_\nu$, while (2) follows directly from the smallness of $\nu$ and the fact that supp$(f)\subset 2Q_1\cap \Sigma$. Next, observe that since $A_1$ is a relatively open set in $\Sigma$, we have that for every $\bx\in \Sigma$,
\begin{align*}
 1_{A_1}(\bx)\leq \underset{\nu\to 0}\liminf \iint_{\Sigma} \Lambda_\nu(\bx,\bz)1_{A_1}(\bz)d\sigma(\bz)
 \leq \underset{\nu\to 0}\liminf f_\nu(\bx),   
\end{align*}
by the last inequality in \eqref{est}. Hence (3) holds.

To prove (4), we observe that the second inequality is simply a re-statement of the second inequality in \eqref{est}, so it suffices to show that
\begin{align}\label{3.9}
    ||f_\nu||_{BMO(\Sigma)}\lesssim ||f||_{BMO(\Sigma)}, \qquad \text{uniformly in } \nu.
\end{align}
To this end, we fix a surface cube $\Delta=\Delta(\by,r)$, and we consider two cases.

\noindent{\bf Case 1:} $\nu\geq r.$ In this case, set $c:=\tiltfiint_{\Delta(\bx, 2\nu)} f$, so that by 
ADR, \eqref{eq3.11} and the construction of $\Lambda_\nu$,
\begin{align*}
    \tiltfiint_{\Delta} |f_\nu-c|\,d\sigma \lesssim \tiltfiint_{\Delta}\tiltfiint_{\Delta(\bx, 2\nu)} |f-c|\, 
    d\sigma d\sigma \lesssim ||f||_{BMO(\Sigma)}.
\end{align*}

\noindent{\bf Case 2:} $\nu< r$. In this case, set $c:=\tiltfiint_{2\Delta} f$. Then by Fubini's Theorem,
\begin{align*}
    \tiltfiint_{\Delta} |f_\nu(\bx)-c|\,d\sigma(\bx) \lesssim \tiltfiint_{2\Delta} |f(\bz)-c|\iint_{\Sigma} \Lambda_\nu(\bx,\bz)d\sigma(\bx)d\sigma(\bz) \lesssim ||f||_{BMO(\Sigma)},
\end{align*}
where again we have used ADR, \eqref{eq3.11} and the compact support property of $\Lambda_\nu(\bx,\bz).$

Since these bounds are uniform all over $\by \in \Sigma$, and $r\in (0,\diam(\Sigma))$, we obtain \eqref{3.9}.
\end{proof}

\end{document}